\documentclass[twoside,openright,english,12pt]{amsart}
\usepackage[T1]{fontenc}
\usepackage{lmodern}
\usepackage{geometry} % see geometry.pdf on how to lay out the page. There's lots.
\usepackage{babel} 
\usepackage{color}
\usepackage{enumerate}
\usepackage{mathtools}
\usepackage{amssymb,latexsym}
\usepackage{amsmath,amsthm}
\newtheorem{theorem}{Theorem}[section]
\newtheorem{lemma}[theorem]{Lemma}
\newtheorem{corollary}[theorem]{Corollary}
\newtheorem{proposition}[theorem]{Proposition}
\theoremstyle{definition}
\newtheorem{definition}[theorem]{Definition}%[section]

\newtheorem{remark}{Remark}[section]

\def\C{\mathbb{C}}
\def\R{\mathbb{R}}

\def\Z{\mathbb{Z}}

\def\cC{\mathcal{C}}
\def\cF{\mathcal{F}}
\def\cH{\mathcal{H}}
\def\cI{\mathcal{I}}
\def\cO{\mathcal{O}}
\def\cP{\mathcal{P}}
\def\cF{\mathcal{F}}
\def\cT{\mathcal{T}}
\def\cW{\mathcal{W}}

\def\fA{\mathfrak{A}}
\def\fC{\mathfrak{C}}
\def\fg{\mathfrak{g}}
\def\fG{\mathfrak{G}}
\def\fK{\mathfrak{K}}

\theoremstyle{theorem}
\newtheorem*{theorem-non A}{Main Theorem A}
\newtheorem*{theorem-non}{Main Theorem B}

\def\e{\varepsilon}

\title[Pseudo-elliptic geometry $\&$ Maurer--Cartan structures]{Pseudo-elliptic geometry of a class of Frobenius-manifolds \\ $\&$\\  Maurer--Cartan structures}
\author{N. Combe \and Ph. Combe\and H. Nencka}
\address{ Max Planck Institut for Matthematics in Sciences, Inselstra\ss e 22, 04103 Leipzig, Germany}
\email{noemie.combe@mis.mpg.de}

\date{\today}

\makeatletter
\@namedef{subjclassname@2020}{\textup{2020} Mathematics Subject Classification}
\makeatother

\keywords{Non-Euclidean geometry, Paracomplex geometry, Lorentzian manifold, Jordan algebras, Frobenius manifold}
\subjclass[2020]{Primary: 17Cxx, 53B05, 53B10, 53B12; Secondary:16W10, 53D45}

\thanks{This research was supported by the Max Planck Society's Minerva grant. The authors express their gratitude towards MPI MiS for excellent working conditions.}

\begin{document}

\begin{abstract}
The recently discovered fourth class of Frobenius manifolds by Combe--Manin in \cite{CoMa} opened and highlighted new geometric domains to explore. 
The guiding mantra of this article is to show the existence of hidden geometric aspects of the fourth Frobenius manifold, which turns out to be related to so-called {\it causality conditions.} 
Firstly, it is proved that the fourth class of Frobenius manifolds is a Pseudo-Elliptic one.
Secondly, this manifold turns out to be a sub-manifold of a non-orientable Lorentzian projective manifold.
Thirdly, Maurer--Cartan structures for this manifold and hidden geometrical properties for this manifold are unraveled.
 {\it In fine}, these investigations lead to the rather philosophical concept of  causality condition, creating a  bridge between the notion of causality coming from Lorentzian manifolds (originated in special relativity theory) and the one arising in probability and statistics.

%The notion of Frobenius manifolds (resp. F-manifolds) arose as an interaction between topology and quantum physics. In this paper, geometries of the recently discovered fourth class of Frobenius manifolds are investigated. First of all, {\it pseudo-Ellipticity} of the manifold is proven. Secondly, this manifold turns out to be a submanifold of a non-orientable Lorentzian projective manifold. Thirdly, using Cartan's theory of $m$-pairs, and Norden--Shirokov's normalization theory approach, we find {\it Maurer--Cartan structures} for this manifold and unravel hidden geometrical properties. Finally, this work leads to the rather philosophical concept of {\it causality condition}. A bridge between the notion of causality coming from Lorentzian manifolds and the notion of causality arising in probability and statistics is established. 
\end{abstract}

\maketitle

%\vspace{5pt}
%\setcounter{tocdepth}{1}
\tableofcontents

\section*{Introduction}
The notion of Frobenius manifolds (resp. $F$-manifolds) is the fruit of fifty years of remarkable interaction between topology and quantum physics. This relation involves the most advanced and sophisticated ideas on each side, and lead to Topological Quantum Field Theory.

Three classes of Frobenius manifolds include quantum cohomology (topological sigma-models), unfolding spaces of singularities (Saito’s theory, Landau-Ginzburg models), and Barannikov--Kontsevich construction starting with the Dolbeault complex of a Calabi--Yau manifold and conjecturally producing the B-side of the mirror conjecture in arbitrary dimension \cite{Ma99}.

The recently discovered fourth class of Frobenius manifolds by Combe--Manin in \cite{CoMa} opened and highlighted new geometric domains to explore. The aim of this article is to show the existence of {\it hidden geometric aspects} of the fourth Frobenius manifold. 

This class of Frobenius manifolds includes manifolds of probability distributions, given by a triple $(M,g,\nabla)$ where $M$ is a manifold, equipped with a Riemannian metric $g$ and a torsion free connection $\nabla$. 

As it was shown in \cite{CoCoNen1,CoMa}, the $F$-manifolds in the sense of Combe--Manin \cite{CoMa} have an unexpected meaning from the point of view of algebraic geometry. To give an example, the tangent space to this manifold can be interpreted as a module over a certain unital associative, commutative rank 2 algebra. This algebra being a non division algebra, implies extremely rich non Euclidean geometrical properties and leads to unexpected developments, such as in \cite{CoMaMa1,CoMaMa2}. 

\smallskip 

These geometric investigations lead unexpectedly to the rather philosophical notion of {\it causality}. It turns out that, using our construction, a bridge between the notion of causality coming from Lorentzian manifolds given by S. W. Hawking and G. F. R Ellis in \cite{HaEl} and the notion of causality, arising in probability and statistics \cite{Sha96} can be established,. 

\smallskip 

 According to B. Dubrovin (\cite{D96} and \cite{Ma99}), the main component of a Frobenius structure
 on $M$ is a {\it (super)commutative, associative and bilinear over constants
 multiplication $\circ : \cT_M \otimes \cT_M \to \cT_M$  on its tangent sheaf $\cT_M$}.
 
 Additional parts of the structure in terms of which  further restrictions upon $\circ$
 might be  given, are listed below:
 \begin{list}{--}{}
\item A subsheaf of flat vector fields $\cT^f_M \subset \cT_M$ consisting of
 tangent vectors flat in a certain affine structure.
 
\item A metric (nondegenerate symmetric quadratic form) $g: S^2(\cT_M)\to \cO_M$.
 
\item  An identity $e$. 
 
\item  An Euler vector field $E$ .
\end{list}

We start, as above, with a family of data 
\begin{equation}\label{E:FM}
(M; \quad \circ : \cT_M\otimes \cT_M \to \cT_M; \quad \cT_M^f \subset \cT_M; \quad
g : S^2(\cT_M) \to \cO_M),  
\end{equation}
mostly omitting  identity $e$ and Euler field $E$.
\vspace{3pt}

The main additional structure bridging these data together is a family of (local) {\it potentials} $\Phi$
(sections of $\cO_M$)
such that  for any (local) flat tangent fields $X,Y,Z$ we have
\[
g(X\circ Y,Z) = g(X, Y\circ Z) =(XYZ)\Phi .
\]
If such a structure exists, then (super)commutativity and associativity of $\circ$
follows automatically, and we say that  the family~\eqref{E:FM} defines a {\it Frobenius manifold.}

However, relying on the algebraic and geometric developments of \cite{CoCoNen1,CoMa}, it has been shown that we can consider this space of probability distributions as a manifold of 0-pairs i.e. the space of $(n-1)$-hyperplanes lying in a vector space of dimension $n$.  The theory of $m$-pairs was introduced by E. Cartan in \cite{Ca31}, A.P Shirokov\, \cite{Sh87,Sh02} and A.P. Norden\, \cite{No76}. The interpretation of the fourth class of Frobenius manifold as a manifold of 0-pairs has been proved in \cite{CoMa} in Proposition 5.9. 

\smallskip 

Relying on this fundamental result, leads to proving that the fourth class of Frobenius manifolds are projective (non Euclidean) non orientable manifolds which have the structure of a Lorentzian manifold.  

Throughout the paper we will prove these two following statements:
\begin{theorem-non A}
The class of the fourth Frobenius manifolds is a pseudo-Elliptic manifold of type $S^n_1$ equipped with the following {\it Maurer--Cartan structures}:
\[
\left\{\begin{aligned}
d{\bf r} &= \omega {\bf r}  + \omega^sX_s,\\
dX_i &= \omega_i {\bf r} + \omega_i^jX_j.
\end{aligned}\right.\]
\end{theorem-non A}

\begin{theorem-non}
The fourth class of Frobenius manifolds satisfies the following geometric properties: 

\begin{enumerate}
\item it is a projective non-orientable submanifold of a Lorentzian projective manifold;
\item  it is equipped with a non Euclidean geometry; 
\item this manifolds is uniquely determined by an orientable 2-sheeted cover; 
\item  it is decomposed into two  domains, which are symmetric about a real (projective) Hermitian hyperquadric, being the symmetry hyperplane mirror.
\end{enumerate}
\end{theorem-non}
The method of the proof is essentially based on the Cartan's theory of $m$-pairs and Norden--Shirokov's normalization theory, which from an other point of view implies that Maurer--Cartan structure equations are used. The paper is decomposed as follows.

\medskip

\subsubsection{Plan of the paper}\

$\ast$ In section 1, we recall fundaments on manifolds of probability distributions, as well as differential manifolds defined over an associative, commutative algebra of finite dimension. Since it has been shown that the fourth Frobenius manifold is related  to the geometry over the algebra of paracomplex numbers (see \cite{CoMa}), we focus in particular on the rank 2 algebra of paracomplex numbers. 

\smallskip 

$\ast$ In section 2, we develop the geometry of $m$-pairs, applied to the manifold of probability distributions. By $m$-pair, we mean a pair consisting of an $m$-plane and an $(n-m-1)$-plane in a space of dimension $n$. This relies on Cartan's $m$-pairs theory and Norden--Shirokov's normalisation theory. It leads to establishing a bridge between paracomplex geometry, projective geometry and the manifold of probability distributions.

\smallskip 

$\ast$ In section 3, we show that the class of the fourth Frobenius manifold can be considered as pseudo-Elliptic manifolds. The link between section 3 and section 2 relies on a theorem of Rozenfeld\, (see section 4.1.1 in \cite{Roz97}). From this it follows that class of fourth Frobenius manifolds are Lorentzian. Maurer--Cartan structures for those manifolds are presented.
 
 \smallskip 
 
$\ast$ Section 4 considers the automorphism group of the manifold of probability distributions. We complete a theorem by Wolf \cite{Wo}. This theorem classifies a Riemannian symmetric space, the associated Lie group and its totally geodesic submanifolds. The theorem was stated only for the field of real, complex, quaternion, octonion numbers and was left open for algebras. We complete the result for the algebra of paracomplex numbers. In this section we prove that this specific manifold has a pair of totally geodesic submanifolds, which are isomorphic to the cartesian product of real projective spaces. 

 \smallskip 
 
$\ast$ Section 5 is a conclusion and highlights the connection to causality problems. In particular, it raises questions of bridging causality notion  investigated by Hawking, Ellis and the causality notion from probability and statistics. 

\vfill\eject

\subsection*{List of notations}\ 

\vspace{5pt}
\begin{tabular}{l c l}
$A$ & & algebra of finite dimension \\
$A E^n$ & &  affine space over the algebra $A$\\
$A\frak{M}^{m}$ & &  $m$-module defined over algebra $A$\\
$\frak{A}$ & &  spin factor algebra\\
$B_l^{n}(x,y)$ & &  bilinear form of signature $(l,n-l)$\\
$\cC$ & &  cone in $\mathcal{W}$ of (strictly) positive measures \\
$\fC$& &  algebra of paracomplex numbers\\
$\fC E^m$ & &  paracomplex affine space\\
$\fC K$ & & unitary paracomplex space\\
$\fC\cP^n$ & &  projective paracomplex space\\
$E^{m}$ & &  $m$-dimensional real linear space.\\
 $\mathcal{I}$& &  Ideal\\
$\frak{M}^{m}$ & &  $m$-module \\
$R^n_l$ & &  pseudo-Euclidean space of index $l$\\
$S^n_l$ & &  pseudo-Elliptic space of index $l$\\
$\mathcal{W}$ & &  linear space of signed measures \\
 & & with bounded variations, vanishing on an ideal $\mathcal{I}$ \\
 $X_m$ & &  $m$-dimensional hypersurface in affine space\\

\end{tabular}

\medskip

%%%%%%%%%% 
\section{Overview on the paracomplex geometry}

In this section we present several building blocks to construct manifolds over a finite unital commutative associative algebra. 

\subsection{Fifth Vinberg cone}
Consider $(X,\cF)$ the measure space, as defined in the annexe section. Let $\cI$ be an ideal of the $\sigma$--algebra $\cF$, on which measures vanish. Let $\mathcal{W}$ be a linear space of signed measures with bounded variations, vanishing on the ideal $\mathcal{I}$ of the $\sigma$-algebra $\mathcal{F}$. Let $\cC$ be a cone in $\mathcal{W}$ of (strictly) positive measures on the space $(X,\cF)$, vanishing only on an ideal $\cI$ of the $\sigma$--algebra $\cF.$ We have the following result, bridging the cone $\cC$ and the fifth Vinberg cone.   

 \begin{theorem}\label{T:+++}
The positive cone $\cC$, defined above, is a Vinberg cone, defined over the algebra of paracomplex numbers.
\end{theorem}

This follows from \cite{CoMa}. 

\begin{remark} 
The $n$-dimensional Vinberg cones are in bijection with (semi-simple) Jordan $n$-dimensional algebra. 
Moreover, in the work of Vinberg, there were introduced left symmetric algebras on the convex homogeneous cones. This algebraic structure was given the name Vinberg algebras.

 These algebras are also referred as {\it pre-Lie algebras}, in domains concerned by Hochschild cohomology problems \cite{G} and operads. In Gerstenhaber's works, the Lie bracket involved in the Gerstenhaber structure on the Hochschild cohomology comes from a pre-Lie algebra structure on the cochains. 
\end{remark}

\smallskip 

In this context a Vinberg cone $\cC\subset \cW$ is a non--empty subset, closed with respect to addition and multiplication by positive reals. A convex cone $\cC$  in a vector space $\cW$ with an inner product has a dual cone $\cC^*=\{a\in \cW:\,  \forall b\in \cW,\, \langle a,b\rangle>0 \}$. The cone is self-dual when  $\cC=\cC^*$. It is homogeneous when to any points $a,b\in \cC$ there is a real linear transformation $T:\cC\to \cC$ that restricts to a bijection $\cC\to \cC$ and satisfies $T(a)=b$. Moreover, the closure of $\cC$ should not contain a real linear subspace of positive dimension.

\smallskip 

The automorphism group $\fG$ of this cone forms a real, solvable Lie subgroup of $GL_{n}(\R)$ and the action is simply transitive. The cone is invariant under this Lie group. One can easily pass to the corresponding Lie algebra. This Lie algebra inherits the property of being real and solvable. It splits into an abelian Lie algebra and a nilpotent one. One can establish a bijection between the Vinberg cone, the Lie group, the corresponding Lie algebra and the Vinberg algebra. 

\smallskip 

\subsection{Two dimensional unital algebras}
As a necessary tool towards the understanding of the cone $\cC$ and its related objects, we start with algebraic definitions. Namely, the definition of the rank 2 algebras, including the spin factor algebra.

Consider the unital and bi-dimensional algebra $\fA$, defined by following relations:

\[e_{i}\cdot e_{j}=\sum_{k}C^{k}_{ij}e_{k}\quad \text{with}\quad C^{k}_{ij}=C^{k}_{ji}.\]

The classification of 2-dimensional unital algebras, generated by generators $\{1, \e \}$, splits into three main classes: 
 
\begin{equation}
\e^2=\left\{\begin{aligned}
1&\quad \text{in the paracomplex case, denoted:}\quad \fC\\
-1&\quad \text{in the complex case, denoted:} \quad \C \\ 
0 & \quad \text{in the dual case, denoted:} \quad C_0 \end{aligned}\right.
\end{equation}

\subsubsection{1. Paracomplex numbers} The algebra of paracomplex numbers given by $\langle 1,\e \, |\, \, \e^2=1\rangle$ can be defined, after a change of basis, by a pair of new generators such that:
 \begin{equation}e_{-}= \frac{1-\e}{2},\quad e_{+}=\frac{1+\e}{2}.\end{equation}
These generators have the following relations:
 \[ e_{-}\circ e_{-}=e_{-},\quad e_{+}\circ e_{+}=e_{+},\quad  e_{-}\circ e_{+}=0,\]
 \[e_{-}+ e_{+}=1,\quad e_{-}- e_{+}= \e.\]
we call this new basis a canonical basis. Notice that this new basis highlights the existence of a {\it pair of idempotents} i.e. $e_{-}^2=e_{-}$ and $e_{+}^2=e_{+}$. This algebra has 0 divisors, being different from 0: it is not a division ring. Paracomplex numbers can be associated to the coordinate ring $\fC=\R[x]/(x^2-1)$.

The structure constants $C^{k}_{ij}$ are:
\begin{equation}C_{11}^{1}=C_{12}^2=C_{22}^1=\, 1, \end{equation}
the other structure constants are null.

This semi-simple algebra is isomorphic to $\mathbb{R}\oplus \mathbb{R}$. It can be identified to $\mathbb{R}^{2}$, as a set, but not as an algebra. 

Note that the algebra of paracomplex numbers is also known as the spin factor algebra (see p.4 in \cite{MC04}).
The conjugation operation is as follows. For a given paracomplex number $z=x+\e y$, we have its conjugated version written as $\overline{z}=x-\e y$.

\smallskip 

\subsubsection{2. Complex numbers} For complex numbers, the algebra is given by $\{1,\e\, |\, \e^2=-1\}$. There is no pair of idempotents, contrarily to the paracomplex case. 
The structure constants $C^{k}_{ij}$ are:
\[C_{11}^{1}=C_{12}^2=\, 1,C_{22}^1=\, -1 \]
the other structure constants are null.

\smallskip 

\subsubsection{3. Dual numbers} For dual numbers, the algebra is given by $\{1,\e\, |\, \e^2=0\}$. The structure constants $C^{k}_{ij}$ are:
\[C_{11}^{1}=C_{12}^2=\, 1, \] 
the other structure constants are null. This is a nilpotent algebra but not a division algebra. 

\begin{remark}
Note that the algebra of paracomplex numbers is not a division algebra. 
This has consequences on the geometry of the affine space defined over this algebra. However, note that neither paracomplex numbers nor complex numbers are nilpotent algebras. 
\end{remark}

%%%%%%%%Module over the spin factor algebra
\subsection{Module over the spin factor algebra}
In this section, we rely on the following trilogy, connecting the following algebraic and geometric objects:

\[\Bigg\{n-\text{Algebra}\quad A \Bigg\}  \leftrightarrow \Bigg\{m-\text{Module over algebra}\Bigg\} \leftrightarrow \Bigg\{mn-\text{Vector space} \Bigg\}.\]

In the same vein, a bigger step allows to relate the algebra  $A$ to a manifold $M^{mn}$ over $A$ (it is discussed in the section \ref{S:CR}). 

Let us denote by $A$ the real $n$-algebra, $\frak{M}^{m}$ the $m$-module and $E^{nm}$ the $nm$-dimensional vector space. 
If $A$ is a real $n$-algebra with basis elements $e_a$, the linear space (or free module) $A\frak{M}^{m}$ admits a {\it real interpretation} in the space $\frak{M}^{mn}$. 

In this interpretation, each vector $x = \{x^i\}$ in $A \frak{M}^{m}$, with coordinates $x^i = x^{ia}e_a$, is interpreted as the vector $\hat{x}=\{x^{ia}\}$ in $\frak{M}^{mn}$. If we replace in this definition the real linear space $\frak{M}^n$ by a linear space or free module $A\frak{M}^n$, we obtain the affine space $A E^n$ over the algebra $A$. In the space $A E^n$ the affine coordinates of points straight lines, planes, $m$-planes, and hyperplanes are defined in the same way as in $E^n$ (see \cite{Roz97}, section 2.1.2).

\smallskip 

Let $\frak{A}$ be the  spin factor algebra. Historically, the spin factor algebra on the space $\R\oplus\R^n$ for $n \geq 2$, 
is equipped with the Jordan product on pairs ${\bf x} \bullet {\bf y} = (x_0y_0+ x\cdot y, x_0y+y_0x)$ where $\cdot$ denotes the usual dot product on $\R^n$ and ${\bf x}=(x_0,x)\in \R\oplus\R^n$. In its matricial representation version, the trace form---which is the usual matrix trace---is in this case $Tr({\bf x})=x_0$.

Construct the $m$-module over the spin factor algebra $\frak{A}\frak{M}^{m}$.
The affine representation of the algebra $\frak{A}$, or free module  $\frak{A}\frak{M}^{m}$, admits a realisation in the real linear space $E^{2m}$. We develop this point of view below. 

Let $E^{2m}$ be a $2m$-dimensional real linear space. {\it A paracomplex structure} on $E^{2m}$ is an endomorphism $\fK: E^{2m} \to E^{2m}$ such that $\fK^2=I$. The eigenspaces $E^{m}_+, E^{m}_-$ of $\fK$ with eigenvalues $1,-1$ respectively, have the same dimension.  The pair $(E^{2m},\fK)$ will be called a {\it paracomplex vector space.} We define {\it the paracomplexification of $E^{2m}$} as $E^{2m}_\fC = E^{2m} \otimes_{\R} \fC$ and we extend $\fK$ to a $\fC$-linear endomorphism $\fK$ of $E^{2m}_\fC$.

\begin{lemma}[\cite{Roz97}]
Let $E^{2m}_\fC = E^{2m} \otimes_{\R} \fC$ be endowed with an involutive $\fC$-linear endomorphism $\fK$ of $E^{2m}_\fC$. Then,  the space $E^{2m}_\fC$  is decomposed into the direct sum of a pair of $m$-dimensional subspaces $E^{m}_{+}$ and $E^{m}_{-}$ such that: \[E^{2m}_\fC=E^{m}_{+}\oplus E^{m}_{-},\] verifying:
\[
E^{m}_{+} = \{v\in E^{2m}_\fC \, |\, \fK v=\e v\}=\{v+\e\fK v\, |\, v \in E^{2m}_\fC\},
\]
\[
E^{m}_{-} = \{v\in E^{2m}_\fC |\,  \fK v= -\e v\}=\{v-\e\fK v\, |\, v\in E^{2m}_\fC\}.
\]
\end{lemma}

\begin{remark}
This splitting also appears in the context of a space over complex numbers. 
\end{remark}

\subsection{Paracomplex manifold}\label{S:CR}
We establish the final building block relating the unital real algebra $A$, the $A$-module and the manifold defined over $A$.  
Let $y=f(x)$ be a (analytic) function, whose domain and range belong to a commutative algebra (i.e. $C_{jk}^h=C_{kj}^h$). We put $x=\sum_ix_ie_i,$ $y=\sum_iy_ie_i.$
From the generalized Cauchy--Riemann we have the following:
\begin{equation}\sum_h\frac{\partial y_i}{\partial x_h}C_{jk}^h=\sum_h\frac{\partial y_h}{\partial x_i}C_{hk}^j,\end{equation}
where $C_{jk}^h$ are the constant structures (see \cite{Sha66}). 

\vskip.2cm
We restrict our attention to the case of paracomplex manifolds.
A {\it paracomplex manifold} is a real manifold $M$ endowed with a paracomplex structure $\fK$ that admits an atlas of paraholomorphic coordinates (which are functions with values in the algebra $\fC = \R + \e\R$ defined above), such that the transition functions are paraholomorphic.

Explicitly, this means the existence of local coordinates $(z_+^\alpha, z_-^\alpha),\, \alpha = 1\dots, m$  such that
paracomplex decomposition of the local tangent fields is of the form
\begin{equation}
T^{+}M=span \left\{ \frac{\partial}{\partial z_{+}^{\alpha}},\, \alpha =1,...,m\right\} ,
\end{equation}
\begin{equation}
T^{-}M=span \left\{\frac{\partial}{\partial z_{-}^{\alpha}}\, ,\, \alpha =1,...,m\right\} .
\end{equation}
Such coordinates are called {\it adapted coordinates} for the paracomplex structure $\fK$.

By abuse of notation, we write $\partial_z$ instead of $\frac{\partial}{\partial z^{\alpha}}$.

%%%%%%%
  We associate with any adapted coordinate system $(z_{+}^{\alpha}, z_{-}^{\alpha})$ a paraholomorphic coordinate system $z^{\alpha}$ by 
\begin{equation}
z^\alpha\, =\, \frac{z_{+}^{\alpha}+z_{-}^{\alpha}}{2} +\e\frac{z_{+}^{\alpha}-z_{-}^{\alpha}}{2}, \alpha=1,...,m .
\end{equation}

We define the paracomplex tangent bundle as the $\R$-tensor product $T^\fC M = TM \otimes \fC$ and we extend the endomorphism $\fK$ to a $\fC$-linear endomorphism of $T^\fC M$. For any $p \in M$, we have the following decomposition of $T_{p}^\fC M$:
\begin{equation}
T_p^\fC M=T_p^{1,0}M \oplus T_p^{0,1}M\,
\end{equation}
where 
\begin{equation}
T_p^{1,0}M = \{v\in T_p^\fC M | \fK v=\e v\}=\{v+\e \fK v| v \in E^{2m}\} ,
\end{equation}
\begin{equation}
T_p^{0,1}M = \{v\in T_p^\fC M | \fK v= -\e v\}=\{v-\e \fK v|v\in E^{2m}\}
\end{equation}
are the eigenspaces of $\mathfrak{K}$ with eigenvalues $\pm \e$.
The following paracomplex vectors 
\begin{equation}
\frac{\partial}{\partial z_{+}^{\alpha}}=\frac{1}{2}\left(\frac{\partial}{\partial x^{\alpha}} + \e\frac{\partial}{\partial y^{\alpha}}\right),\quad \frac{\partial}{\partial{z}_{-}^{\alpha}}=\frac{1}{2}\left(\frac{\partial}{\partial x^{\alpha}} - \e\frac{\partial}{\partial y^{\alpha}}\right)
\end{equation}
form a basis of the spaces $T_p^{1,0}M$ and $T_p^{0,1}M$. 

This paragraph presents the final building block of the construction of the manifold over paracomplex numbers. Hence, now it is possible to present the  proof of the Main Theorem A.
%%%%%%%%%m-pairs

\section{First part of the proof of Theorem A}
In view of proving the Main Theorem A, we give a first part of the proof, by introducing the following construction of the  paracomplex projective spaces and related notions.
\subsection{Construction of a paracomplex projective space}
Let us introduce a notion of projective spaces over paracomplex numbers $\fC\cP^n$. The notation  $\fC E^n$ stands for paracomplex affine space. 
 Let $\R\cP^n$ be a real $n$- dimensional projective space. Any point of the  $\R\cP^n$ space can be determined by a system of homogeneous coordinates $[X^0: X^1: ...: X^n]\in \R\cP^{n}$.

\smallskip 

Points of $\fC\cP^n$ the projective paracomplex space are given by the homogeneous coordinates $[X^0: X^1:...:X^n]\in \fC\cP^n$, where $X^i$ are paracomplex numbers. In such a space $\fC\cP^n$ it is possible to define straight lines, planes and hyperplanes as in the real projective space $\R\cP^n$. 
In that space, any point can be given by $n$ real numbers $x^1, x^2,..., x^n$ which are paracomplex numbers. 

Exactly as in the case of real projective spaces,  one defines the relation between the real affine space and the real projective space $\R\cP^n$ by considering that to any point $x^i$ one can attach  homogeneous coordinate $X^i/X^0$, one can proceed similarly for the affine and projective paracomplex space. To any point of affine paracomplex space $\fC E^n$, given by coordinates $x^i$ (which are paracomplex numbers), there corresponds a homogeneous coordinate $X^i/X^0$ in the paracomplex projective space $\fC\cP^n$. As it is known, in the classical (real or complex) framework, the points at infinity in projective space are given by hyperplanes at infinity i.e. we have an equation of the type $X^0=0$. However, in the case of paracomplex projective spaces, there exist not only infinity hyperplanes but as well {\it special points}: the points corresponding to the {\it zero divisor}.

From another point of view, the paracomplex projective space $\fC\cP^n$ and the real projective space $\R\cP^n$ are directly related. The points of 
$\fC\cP^n$ can be realized as pairs of points lying in the real projective space  $\R\cP^n\times \R\cP^n$. Indeed, a given point $X^i$ of $\fC\cP^n$ is given by a pair of points in  $\R\cP^n\times \R\cP^n$, where 
\[X^i=x^i e_+ + y^ie_-,\] and these real projective points have for coordinates $x^i$ and $y^i.$

Let us consider a transformation of the coordinates $x^i$ and $y^i$ into the coordinates $kx^i$ and $ly^i.$ Then, the coordinates $x^i$ are replaced by
the coordinates $(ke_+ + le_{-})x^i.$ Under this transformation, a straight line  is transformed into a pair of straight lines and the (hyper)planes are transformed into pairs of (hyper)planes, each of which lie respectively in a copy of $\R\cP^n$.

From these transformations, it follows that in the paracomplex projective space $\fC\cP^n$ one has the following property:
\smallskip 

 {\it Through two points can pass an infinite number of straight lines. Moreover, two straight lines of $\fC\cP^n$ in the same plane will intersect themselves in more than one point.}

\subsubsection{Projective Group of transformations}
In the $\fC\cP^n$ space one can define a group of collineations, correlations, resp. anti-collineation and anti-correlations.
The group of collineations is determined by the following system of equations:

\begin{equation}\label{E:sys}
k{X’}^i= \sum_j a^i_j X^j, \quad i,j= 1,2,\cdots,n.
\end{equation}
where $a^i_j$ are matrix entries, being paracomplex numbers, $k$ is an arbitrary paracomplex number. 
The anti-collineation is obtained by taking the conjugation of all the $X'^j$ and $X^i$ in  (\ref{E:sys}).

\medskip 
\subsection{Some notions on $m$-pairs}
We have previously recalled rudiments of paracomplex geometry and showed the relation to the fourth Frobenius manifold. 
In this section, we introduce the language of $m$-pairs. This allows us to show the connection to projective geometry and later on to the pseudo-Elliptic spaces.  
\smallskip
We define  $A\cP^{n}$ to be the $n$-dimensional projective space defined over an algebra $A$ (associative, commutative, unital of finite dimension). By abuse of notation and whenever the context is clear we will use simply the notation $\cP^{n}$. Let $X_{d}$ be a $d$-dimensional surface of the $n$-dimensional projective space $A\cP^{n}$, with $d\leq n$. 

\begin{definition}[Normalized surface]
The surface $X_{d}$ is said to be normalized in the Norden sense, if at each point $p\in X_{d}$, are associated the two following hyperplanes:
\begin{enumerate}
\item Normal of first type, $P_{I}$, of dimension $n-d$, and intersecting the tangent $d$-plane $T_{p}X_{d}$ at a unique point $p \in X_{d}$. 
\item Normal of second type, $P_{II}$, of dimension $d-1$, and included in the $d$-plane $T_{p}X_d$, not meeting the point $p$.
\end{enumerate}
\end{definition}

This decomposition expresses the duality of projective space. In particular, in the limit case, where $d=n$, then $P_{I}$ is reduced to the point $p$ and $P_{II}$ is the $(n-1)$-surface which does not contain the point $p$. 
This property is nothing but the usual duality of projective space. Note that in this case, $X_{n}$ can be identified with the projective space $\cP^{n}$.

\begin{definition}\label{D:mpairs}
A pair consisting of an $m$-plane and an $(n-m-1)$-plane is called an $m$-pair. 
\end{definition}

\smallskip 

We establish a relation to Grassmannians. A Grassmannian $G(k,n)$ is a space that parametrizes all $k$-dimensional linear subspaces of the $n$-dimensional vector space. 
In particular, the Grassmannian $G(1,n)$ is the space of lines through the origin in the vector space and is the same as the projective space $\cP^{n-1}$.

\smallskip 

\begin{remark}
Each $m$-pair corresponds to a point of a Grassmannian of type $G(m,n)$. Reciprocally, every point in the  Grassmannian manifold $G(m,n)$ defines an $m$-plane in $n$-space. Fibering these planes over the Grassmannian one arrives at the vector bundle, which generalizes the tautological bundle of a projective space. Similarly the $(n-m)$-dimensional orthogonal complements of these planes yield an orthogonal vector bundle. 
\end{remark}

From~\cite{No58,Sh87}, for normalized surfaces associated to an $m$-pair space, the following properties holds:
\medskip
\begin{lemma}\label{L:pairs}
\ 

\begin{enumerate}
\item The space of $m$-pairs is a projective, differentiable manifold.
\item For any integer $m\geq 0$, a manifold of $m$-pairs contains 2 flat, affine and symmetric connections. 
\end{enumerate}
\end{lemma}
Turning our attention to 0-pairs, an important key lemma relates 0-pairs, projective spaces and Grassmannians.

\subsection{Paracomplex projective geometry and the fourth Frobenius manifold}
We have the following:
\begin{lemma}[Key lemma]\label{R:1} 
The space of $0$-pairs can be identified with an $(n-1)$-dimensional projective space. This is a Grassmannian space of type $G(1,n)$. 
\end{lemma}

\begin{proof}
For $m=0$, a $0$-pair consists of a point and of an $(n-1)$-hyperplane. So, this amounts to considering the space of $(n-1)$-hyperplanes in an $n$-dimensional space. In other words, this is a Grassmannian of type $G(n-1,n)$. We have a (non-canonical) isomorphism of $G(n-1,n)$  and $G(1,n)$. This isomorphism of Grassmannians sends an $(n-1)$-dimensional subspace into its $1$-dimensional orthogonal complement. Since $G(1,n)$ is the same as the projective space $\cP^{n-1}$, therefore,  we can identify 0-pairs to an $(n-1)$-dimensional projective space. 
\end{proof}

This lemma plays an important role, in particular in relation to the next proposition.  
\begin{proposition}\label{P:0-pairs}
Suppose that $(X,\cF)$ is a finite measurable set where the dimension of $X$ is $n+1$, and measures vanish only on an ideal $I$. Let $S$ be the space of probability distributions on $(X,\cF)$. Then, the space $S$ is a manifold of 0-pairs.
\end{proposition}
\begin{proof}
The statement corresponds to Proposition 5.9 in \cite{CoMa}.\end{proof}

We have introduced the previous part, in order to discuss the statistical manifolds. In particular, manifolds of probability distributions are related to the $m$-pairs in the following way: 

\begin{corollary}\label{C:proj}
The fourth Frobenius manifold is identified with the paracomplex projective space $\fC\cP^{n}$.
\end{corollary}

\begin{lemma}\label{T:3}
The manifold of probability distributions $S$ has a pair of flat, affine, symmetric connections. 
\end{lemma} 

\begin{proof}
There are different ways of proving this. One possibility is that this follows from the calculation in \cite{BuNen1,BuNen2}.
Another, and more geometric, approach is to apply Lemma\, \ref{R:1} and Lemma\, \ref{L:pairs}.\end{proof}

\begin{remark}
In other words, the fibration is done with the algebra of 2 connections.
\end{remark}

Let us recall the following proposition:
\begin{proposition}\label{P:isome}
The space of $0$-pairs in the projective space is isometric to the hermitian projective space over the algebra of paracomplex numbers.
\end{proposition}
\begin{proof} 
see e.g.~\cite{Roz97} section 4.4.5.
\end{proof}

Finally, from Proposition \ref{P:isome} and Proposition \ref{P:0-pairs} it follows that: 

\begin{proposition}
The statistical manifold is isometric to the hermitian projective space over the algebra of paracomplex numbers.
\end{proposition}

As a last point bridging the statistical manifold and the paracomplex space, we have that:

\begin{lemma}
Suppose that $(X,\cF)$ is a finite measurable set where the dimension of $X$ is $n + 1$, and measures vanish only on an ideal $I$. The space $S$ of probability distributions on $(X,\cF)$ is isomorphic to the hermitian projective space over the cone $M_+(2, \C)$.
\end{lemma}
\begin{proof} 
See Theorem 5.10 in \cite{CoMa}.
\end{proof}

%%%%%%%%

\section{Second part of the proof of Theorem A}
In this section, Maurer--Cartan structures for the fourth Frobenius manifold are presented. These considerations arise from our next result bridging manifold of probability distributions and so-called {\it pseudo-Euclidean spaces.}

\subsection{Maurer--Cartan structures for the fourth Frobenius manifolds}
 Pseudo-Euclidean spaces, denoted $R^n_l$, arise from the modification of one of the axioms of classical Euclidean spaces $\R^n$. The fifth axiom turns into the following:

{\it ``There are $l$ mutually orthogonal vectors $v_a$ with negative inner squares $v^2_a$ and $n-l$ mutually orthogonal vectors $v_u$  with positive inner squares $v^2_u$, and each vector $v_a$ is orthogonal to each vector  $v_u$.''}
 
The space $R^n_l$ is a pseudometric space, and the integer $l$ is called the index of this space. 

To establish the relation between manifold of probability distributions and pseudo-Euclidean spaces, we use Norden's normalisation theory \cite{No58,No76,Sh87}. More precisely, this evolves around {\it structural equations} of an affine connection space. The Norden method goes as follows.  

Let $X_{m}$ be an $m$-dimensional surface. Equip it with coframes $\{\omega^i\}$ in an affine space $\R E^{n+1}$. 
It is known that there exists an equivalence between an affine connection on $X_{m}$ and an infinitesimal connection in the principal bundle space of linear frames of the manifold $X_{m}$. So, we can establish the structural equations of the affine connection form in the following way:

\begin{equation} d\omega^{i}-\omega^s\wedge \omega_s^i=\Omega^i, \quad d\omega_j^{i}-\omega_j^s\wedge \omega_s^i=\Omega_j^i, \end{equation}
where $\omega_j^i$ are the connection forms and $\Omega^i, \Omega_j^i$ are respectively the torsion and curvature forms of the affine connection.

\smallskip 

Consider a surface $X_m$ and its polar vector equation ${\bf r}={\bf r}(u^1,\dots, u^m)$. Let $M({\bf r})$ be a point on $X_m$ given by the intersection of the line passing through the origin and collinear to the vector  ${\bf r}$. 

Choose a framing of the surface $X_m$ given by $m$ framing vectors $\{{\bf e}_{i}\}_{i=1}^m$ at $M({\bf r})$ and belonging to the $(m + 1)$-dimensional subspace of $\R E^{n+1}$. Consider the normal space to the tangent space, generated by $n-m$ framing vectors, denoted ${\bf e}_\alpha$. The framings are defined as follows $\{{\bf r}, {\bf e}_i, {\bf e}_\alpha\}$ (notice that it includes the polar vector).

These framings verify the following classical system of Maurer--Cartan like equations:
\begin{equation}
\begin{aligned}
d{\bf r} &= \omega {\bf r} + \omega^s{\bf e}_s\\
d{\bf e}_i &= \omega_i {\bf r} + \omega_i^s{\bf e}_s+ \omega_i^\beta {\bf e}_\beta\\
d{\bf e}_\alpha &= \omega_\alpha {\bf r} + \omega_\alpha^s{\bf e}_s+ \omega_\alpha^\beta {\bf e}_\beta
\end{aligned}
\end{equation}

%%%%%
It was proved in \cite{Roz97}, section 4.1.1 and Theorem 4.3 and Theorem 4.1, that a paracomplex projective space is a  pseudo-Euclidean manifold.
%% Dual lemma 

We now consider the geometry of the fourth Frobenius manifold and prove that:
\smallskip

\begin{theorem}\label{T:PE}
The manifold $S$ is a pseudo-Euclidean space $R^n_l$.
\end{theorem}

\begin{proof}
As was shown in Proposition \ref{P:0-pairs}, the fourth Frobenius manifold are defined by 0-pairs and therefore the covectors ${\bf e}_\alpha$ {\it do not} exist in our case.
We now apply the Norden method \cite{No76}. Let the polar vector {\bf r} be, in the manifold of probability distributions context, defined as the so-called affine canonical coordinates (\cite{Sh87}, p.146). Define, the coframe $\omega$ to be the affine connection component defined in \cite{Sh87}. The vectors ${\bf e}_i$ are nothing  but the score vectors, which have been defined in the Appendix under the notation $X_j=\partial_j\ln \rho_{\theta}$, where $\rho_{\theta}$ is a probability distribution. 
Then, the following system of equations is satisfied:
\begin{equation}\begin{aligned}
d{\bf r} &= \omega {\bf r}  + \omega^sX_s\\
dX_i &= \omega_i {\bf r} + \omega_i^jX_j.
\end{aligned}\end{equation}
In this way we define {\it Maurer--Cartan structures} for the fourth Frobenius manifolds. 

By Rozenfeld's theorem, a paracomplex projective space is a pseudo-Euclidean space. Thus, the manifold of probability distributions is a pseudo-Euclidean space and the statement is proven.
\end{proof}

\subsection{Pseudo-Ellipticity}

Let $R^n_l$, $0\geq l\geq n$, be a real coordinate $n$-dimensional vector space with bilinear form: 
\begin{equation}
B^n_l(x,y)=-\sum_{i=1}^l x_iy_i +\sum_{j=l+1}^nx_jy_j.
\end{equation}

In virtue of Proposition \ref{P:0-pairs}, Lemma \ref{R:1}, and Lemma \ref{L:coco}, a manifold of probability distributions is a manifold of 0-pairs and dual to the section $\cH$ of the fifth Vinberg cone. So, 
this implies, that the bilinear form is given by $B^n_1(x,y)=- x_1y_1 +\sum_{j=2}^nx_jy_j
$ and the the statement below follows.

\begin{proposition}
The fourth Frobenius manifold $S$ is a pseudo-Euclidean space $R^n_1$ of index one.
\end{proposition}
\begin{proof}
In order to show that the space has index 1, it was shown by Chentsov\cite{Ce3} that this manifold is geodesically convex and that the maximal submanifolds are totally geodesic. From Wolf's theorem \cite{Wo}, we have that the space $R^n_l$ is pseudo convex only if the index $l$ is
1. 
\end{proof}

More precisely, we can refine our statement by stating that:
\begin{theorem}\label{P:pseudoell}
The manifold $S$ is a real pseudo-Elliptic space $S^n_1$ of index one.
\end{theorem}
\begin{proof}

The manifold $S$ equipped with a Riemannian metric $g$ is a Riemannian manifold $(S, g)$. In order to show that the manifold $S$ is a real pseudo-Elliptic space, it is sufficient to find the metric.  

 Bhattacharyya \cite{Ba}  shows that the distance between two points in a manifold of probability distributions (i.e. distance between a pair probability distributions $P, P^*$), is given by:

\[d(P,P^*) = \int_{\Omega}\sqrt{\rho\rho^*}d\lambda,\]
where $\rho$ and $\rho^*$ are the Radon--Nikodym derivatives of $P$ and $P^*$ respectively w.r.t $\lambda$.
However, it was shown in \cite{BuCoNen99} (section 2 p.89) that the distance $d(P,P^*) $ is given by $\cos^2\omega= \int_{\Omega}\sqrt{\rho\rho^*}d\lambda$.

This coincides with the metric on the pseudo-Elliptic manifold:
 \[cos^2\frac{\delta}{r}= \overline{XY,\alpha\beta}\]
where $X,Y$ are points, $\alpha,\beta$ are points at infinity and $\overline{XY,\alpha\beta}$ is the cross-ratio of these points. The real radius of the curvature is $r$. In particular, the interpretation regarding the manifold of probability distributions is as follows: the points $X, Y$ correspond to the probability distributions and $\alpha,\beta$ correspond to the probability distributions on the boundary of the cone.  
This proves the statement about pseudo-Ellipticity.  
\end{proof}

%%%%%%%Lorentz section 3
\section{Theorem B: the fourth Frobenius manifold is a Lorentzian manifold}
Now, we present the proof of the Main Theorem B.
\subsection{Lorentzian manifolds}
\begin{definition}\label{D:Lorentz}
A Lorentzian manifold is a pseudo Riemannian manifold which is equipped with an everywhere non-degenerate, smooth, symmetric metric tensor $g$ of signature $(1,n-1)$ i.e. such that the bilinear form verifies:
 \[B^n_1(x,y)=- x_1y_1 +\sum_{j=2}^nx_jy_j.\]
\end{definition}

\begin{proposition}\label{P:Lor1}
A pseudo-Euclidean manifold in a pseudo-Euclidean space of index 1 is a Lorentzian manifold. 
\end{proposition}
\begin{proof}
A pseudo-Euclidean space of index 1 $R^n_1$ has a bilinear from \[B^n_1(x,y)=- x_1y_1 +\sum_{j=2}^nx_jy_j
.\] Applying the Definition \ref{D:Lorentz}, we can say that a pseudo-Euclidean manifold of index 1 is a Lorentzian manifold. 
\end{proof}
The cone of positive measures of bounded variations belongs to the fifth class (see \, \cite{CoMa}). This means we have a cone defined over the algebra of paracomplex numbers. This cone is also known in other branches as the time future like cone. In particular, we have that:

\begin{proposition} 
The cone $\cC$ defined over the algebra of paracomplex numbers is a pseudo-Riemannian manifold $(M,g)$. This is a differentiable manifold $M$, equipped with an everywhere non-degenerate, smooth, symmetric metric tensor $g$ of signature $(1,n-1)$. 
\end{proposition}

For a metric $g$, the signature $(1,n-1)$ implies that we have:
\[g=-dx_1^2+dx_2^2+\dots +dx_n^2.\]

In other words:
\begin{corollary} 
The cone $\cC$ is a Lorentzian manifold. 
\end{corollary}
As for the manifold of probability distributions, we have that:
\begin{proposition}\label{C:lo}
The manifold of probability distributions $S$ is a projective Lorentzian manifold.  
\end{proposition}
\begin{proof}
A manifold of probability distributions is a pseudo-Elliptic space $S^n_1$, by Theorem\, \ref{P:pseudoell}. The bilinear form is given by: $B^n_1(x,y)=- x_1y_1 +\sum_{j=2}^nx_jy_j$. Applying Definition\, \ref{D:Lorentz}, the conclusion is straight forward. 
\end{proof}

\subsection{The mirror symmetries of the fourth Frobenius manifold}
A central Hermitian hyperquadric form can be written generally as: %(2.78)
\[xQx+c=\sum_i \overline{x}^ix^i+c=0.\]
Those central Hermitian hyperquadrics are called Hermitian ellipsoids. Hermitian ellipsoids, in the paracomplex space $\fC E^n$, are homeomorphic to the topological product of $\R^n$ and a hypersphere in $\R^n$, denoted $S^{n-1}$. We call the hyperquadric given by the equation $xQx = 0$ an absolute hyperquadric of the pseudo-Elliptic space.

\begin{remark}
The equation $xQx = 0$ is the equation of an oval hyperquadric in $\cP^n$. If the vectors $x$ and $y$ represent points $X$ and $Y$ in this space, the vectors $Qx$ and $Qy$ can be regarded as covectors representing the hyperplanes polar to these points with respect to this oval hyperquadric.
\end{remark}

An $n$-dimensional non-Euclidean Riemann space can be determined as an $n$-dimensional projective space in which the distance between two points $x$ and $y$ is given by: 

\[ \cos^2 \frac{\omega}{r}= \frac{(x,y)}{(x,x)\cdot (y,y)},\]
where $(x,y)$ is a bilinear form, $r$ and $\omega$ are real numbers. 
However, in the case of a projective paracomplex space, the distance between points is given by a slightly different formula.

Historically, the idea to construct new type of non Euclidean space was given by C. Segre. Namely, he proposed to introduce new type of form and then to construct new types of spaces. 

Let us introduce the notion of Hermitian form:
\[\{x,y\}=\overline{x}^0 y^0+ \overline{x}^1 y^1+...+\overline{x}^n y^n,\]
with property the that $\{x,y\}= \overline{\{y,x\}}$.
Therefore, the form $\{x,x\}$ is always real.

Now, we can introduce the distance $\omega$ in the paracomplex projective space as follows:
\[\cos^2 (\frac{\omega}{r}) = \frac{ \{x,y \}\cdot \{y,x\} }{\{x,x\}\cdot \{y,y\}},\]
where $r$ is a radius of the curvature of the space. This space is a so-called {\it unitary paracomplex space} $\fC K^n$.

\smallskip
 
For the convenience of the reader, we recall the definition of a unitary paracomplex space $\fC K^n$.
A vector space over the algebra of paracomplex numbers $\fC$ is a vector space endowed with an inner product $\{ \cdot, \cdot\}$ satisfying the following axioms:

\begin{enumerate}
\item $\{a,b\}=\overline{\{b,a\}}$
\item $\lambda\{a,b\}=\{\lambda a,b\}$, where $\lambda \in \fC$
\item$ \{a+b,c\}=\{a,c\}+ \{b,c\}$
\item If $a\neq 0$ then $\{a,a\}>0$.  
\end{enumerate}

The collineation in the projective paracomplex space $\fC\cP^n$ which conserves the distances between two points we call the actions of the space. The matrices of these collineation are given by:
\[
\overline{a^0_i} a^0_j + \overline{a^1_i} a^1_j +...+\overline{a^n_i} a^n_j =\begin{cases} 
1, &\text{if}\quad i= j,\\
0, &\text{if}\quad i\neq j.
  \end{cases}\]
  
  \begin{remark}
 
 Notice that this is a discrete analog of the inner product defined for a paracomplex Hilbert space.
 \end{remark}

In the projective paracomplex space one can easily see that there exists as well another action called anti-collineation as well conserving the distance between two points. Then the interpretation of the unitary paracomplex space $\fC K^n$ is the following one.

\begin{theorem}\label{T:mirr}
The statistical manifold  can be defined as one of the two domains into which the hyperquadric $xQx = 0$ divides the paracomplex projective space $\cP^n$, where the distance between two points $X$ and $Y$ is given by: \[cos^2\frac{\delta}{r}= \overline{XY,\alpha\beta},\]
where $r$ is the curvature.
\end{theorem}

Here $\overline{XY,\alpha\beta}$ is the cross ratio of these points and their polar hyperplanes with respect to the hyperquadric $xQx =0$, with 
 $x$ being an arbitrary vector in the affine space, representing two points in $\fC\cP^n$.
 \begin{proof}
 This follows from the Theorem\, \ref{T:PE} and the Theorem 4.3, in \cite{Roz97}.
 \end{proof}

\begin{remark}
Note that the elliptic space $S^n$ can be defined as the projective space $\cP^n$ in which we have a specific distance relying on the cross ratio $\overline{XY,\alpha\beta}$ where $X, Y$ are two points and  $\alpha, \beta$ are their polar hyperplanes. This metric is defined as:
 \[cos^2\frac{\delta}{r}=\overline{XY,\alpha\beta}\]
where $ \overline{XY,\alpha\beta}$ is the cross ratio of these points and their polar hyperplanes with respect to the imaginary hyperquadric $x^2 = 0$. 

In this way, we achieve the first part of the proof of Main Theorem B.
\end{remark}

\section{Second part of the proof of Theorem B} 
A Vinberg cone generates a class of Lie groups and of Lie algebras. There exists an automorphism group of the Vinberg cone (see section 1.1).  Consider the positive cone $\cC$ of strictly positive measures on a space $(X,\mathcal{F})$, vanishing only on an ideal $I$ of the $\sigma$-algebra $\mathcal{F}$. In our case, the automorphism corresponds to the parallel transport, which we define below. 

Let $\cW$ be the space of signed measures of bounded variations (i.e. signed measures whose total variation  is bounded, vanishing only on an ideal $I$ of the $\sigma$-algebra $\mathcal{F}$). To any parallel transport $h$ in the covector space $\cW^{*}$ of the space $\cW$ of $\sigma$--finite measures, we associate
 \[f\xrightarrow{h} f+h,\] an automorphism of the cone $\cC$:
\begin{equation}
\mu \xrightarrow{h} \nu,\, \text{ where }\, \frac{d\nu}{d\mu}(\omega)\, =\, \exp (h(\omega)), 
\end{equation}
and $d\nu/d\mu$ is the Radon--Nikodym derivative of the measure $\nu$ w.r.t. the measure $\mu$.
This automorphism is a  non--degenerate linear map of $\cW$ which leaves the cone invariant. 

\smallskip 

Denote by $\fG$ the group of all automorphisms $h$ such that $h\, =\, \ln\, \frac{d\nu}{d\mu}.$ 
The commutative subgroup of all ``translations'' of the cone $\cC$ is a simply transitive Lie group. To this Lie group $\fG$ the associated Lie algebra $\fg$ defines the derivation of the cone. 

\smallskip 

We focus on the geometry of $\cH$, defined in the Appendix. It has the {\it projective geometry} of a pencil of straight-lines of the cone $\cC$. The geodesics on the cone $\cC$ are the trajectories of 1-parameter subgroup of the group $\fG$ and can be written in the following way:

\begin{equation}\label{E:t}
f(\omega;s)=f(\omega;0)\exp\{s.h(\omega)\}.
\end{equation}

The geodesics on $\cH$ need to be logarithmic projections on $\cH$ of these trajectories. They are distinct from $f(\omega;s)$ by a normalization constant. In this way, a geodesic---crossing a given point $p(\omega; 0)$ in a given direction---can be determined by:

\begin{equation}\label{E:t}p(\omega;s)=\frac{1}{a(s)}p(\omega;0)\exp\{s.q(\omega)\},\quad \text{where} \quad a(s)=\int_{\Omega} \exp(s.q(\omega))p(\omega;0)d\mu.\end{equation}
By duality, we can obtain a similar approach to the geodesics of the manifold of probability distributions $S$, which are trajectories of 1-parameter subgroups.

\subsection{Real Interpretation of unitary paracomplex spaces $\fC K_n$.}
Recall that $\fC K_n$ is a  paracomplex vector space on which an inner product of vectors is defined. In the $(2n+1)$-dimensional non Euclidean Riemann space $R^{2n+1}_l$, there exist families of so-called {\it paratactic} congruent straight lines. Each 0-pair is in on-to-one correspondence with a ray of our congruence. These families are $2n$ parametric families of straight lines along of which it is possible to apply a 1-parameter group of motions (translations along the congruent straight lines). These are so-called {\it paratactic transformations}. 

That means, two straight lines of such a congruence have a 1-parameter set of common perpendiculars, being of the same length (and not just one perpendicular!). That is, the distance between these two straight lines is the same (we take the length of the perpendicular as the distance).
This type of congruence exists as well in the non-Euclidean pseudo-Riemannian space $R^{2n+1}_{n+1}$. In this space the bilinear form will be given by:
\[B^{2n+1}_n(x,y)= - x^0 y^0 - x^1 y^1 - ...  - x^n y^n + x^{(n+1)} y^{(n+1)} +... + x^{(2n+1)} y^{(2n+1)}.\]

The unitary paracomplex space $\fC K_n$ is isomorphic to the non Euclidean pseudo-Riemannian space $R^{2n+1}_n$. This ends and gives a conclusion to discussions on statements appearing in Theorem B.

\subsection{An addition to Wolf's theorem}
In this paragraph, we prove an additional result to Wolf's classification theorem in \cite{Wo}. His result holds for the real, complex, octonionic and quaternonic number fields. However, there is a gap in what concerns algebras that are not fields: for example, paracomplex numbers. We remedy to this precise situation.  

Let us go back to the automorphism group of the cone $\cC$.
\begin{lemma}
Let $\fG$ be the Lie group of automorphisms of the Vinberg cone $\cW$. Let $H$ be a subgroup of $\fG$ leaving the manifold $\cH$ invariant. Then, 
$H$ is a Lie subgroup. 
 \end{lemma}

Let $\cH$ be a  non-empty (compact) subspace of the Vinberg cone $\cW$. Consider the subgroup $H$ of $\fG\subset GL_n(\R)$ leaving this $\cH$ invariant. Then, it is known that this subgroup $H$ is compact. 
It remains to apply the Cartan closed subgroup theorem: any closed subgroup $H$ of a Lie group $\fG$ is a Lie subgroup (and thus a submanifold) of $\fG$. 

\begin{corollary}
$H$ is a compact Lie subgroup.  
\end{corollary}
%%%%%%

The manifold of probability distributions is isometric to a pseudo-Riemannian (projective) space. The group of motions is a simple Lie group of type $D_n$ (see \cite{Roz97}), which is associated to a special orthogonal group $SO(1,2n-1)$ group, and of real rank 1. 

Therefore, we have that:
\begin{lemma}
The fourth Frobenius manifold $S$ can be considered as a compact symmetric space of rank 1.
\end{lemma}

\begin{proof} Indeed, this follows from the above arguments: it is a compact Lie group, the group of motions is a simple Lie group of type $D$. Since it is a projective space and of even dimension the only remaining possibility is to have  a compact symmetric space of rank 1.
\end{proof}

\begin{theorem}
Consider the $2n$-dimensional manifold of probability distributions $S$. If $M$ is a totally geodesic submanifold of $S$ then it is a product of real projective spaces $\R\cP^r\times\R\cP^r$, where $1\leq r \leq 2n$.
\end{theorem}

\begin{proof}
The $2n$-dimensional manifold of probability distributions are by the statement (Lemma\, \ref{R:1} and Proposition\, \ref{P:0-pairs}) identified to a paracomplex projective space, which in turn is isomorphic to a pair of real projective spaces of real dimension $2n$. Now, we invoke the following argument of fixed point sets: for any isometry $f : M \to M$, the fixed point set is a totally geodesic submanifold of $M$.
Taking the isometry $f: \fC\cP^{2n}\to \fC\cP^{2n}$ in the paracomplex projective space $\fC\cP^{2n}$ such that $f: (a,b)\to (a,-b)$, where $a=(z_1,\dots, z_{r+1})$ and $b=(z_{r+2},\dots, z_{n+1})$, the set of points $(a,0)$ forms a fixed set under $f$. So, the fixed point set is $\fC\cP^r$. So, in particular, $\fC\mathbb{P}^r$ being isomorphic to $\R\cP^r\times \R\cP^r$ for $1\leq r \leq 2n$ it defines a totally geodesic manifold.  
\end{proof}

\begin{remark}
 This completes Wolf's theorem \cite{Wo}, which considers totally geodesic submanifolds for spaces defined over the fields $\R, \C, \mathbb{O},\mathbb{H}$. Here we define this for the algebra of paracomplex numbers. 
\end{remark}

We have that the manifold of probability distributions is decomposed into a pair of totally geodesic submanifolds. 
A submanifold $N$ of a Riemannian manifold $(M,g)$ is called totally geodesic if any geodesic on the submanifold $N$ with its induced Riemannian metric $g$ is also a geodesic on the Riemannian manifold $(M,g)$.

\begin{corollary}\label{C:sbmfd}
The $2n$-dimensional manifold of probability distributions $S$ has a pair of totally geodesic submanifolds, being real projective spaces.
\end{corollary}
Now, we prove the following result:
\begin{theorem}\label{T:pro}
The manifold $S$ is isomorphic to $\R\cP^n\times \R\cP^n$. 
\end{theorem}

\begin{proof}
Indeed, by Proposition\, \ref{P:0-pairs}, we have that $S$ is a 0-pair. By Lemma \ref{R:1} a 0-pair is a projective space $P^n$. An $n$-dimensional paracomplex projective space is isomorphic to a cartesian product of $n-$dimensional real projective spaces $\R\cP^n\times \R\cP^n$ (see \cite{Roz97}). So, the manifold $S$ is isomorphic to the cartesian product $\R\cP^n\times \R\cP^n$.
\end{proof}
\begin{corollary}\label{}
The manifold $S$ is a projective variety.
\end{corollary}
\begin{proof}

The product of two projective varieties is a projective variety. Since $\R\cP^n$ are projective varieties, the statement follows directly.
 \end{proof}

\begin{proposition}\label{P:Lo}
The manifold of probability distributions is a non-orientable Lorentzian manifold. 
\end{proposition}
\begin{proof}
Let us apply Corollary\, \ref{C:lo} stating that the manifold of probability distributions is a Lorentzian manifold. The dimension of this manifold is even: it is a projective paracomplex space by Lemma \ref{R:1} and Proposition \ref{P:0-pairs}). 

Theorem \ref{T:pro} states that the manifold $S$ is isomorphic to $\R\cP^{2m}\times \R\cP^{2m}$. By topological arguments, an even dimensional real projective space in non-orientable. Now, since the cartesian product of a pair of manifolds $M\times N$ is orientable iff the manifolds $M$ and $N$ are orientable, the conclusion is straight forward. 
\end{proof}

%%%%%%%%%%%%
We now prove the statement:

\begin{theorem}\label{T:nono}
The class of fourth Frobenius manifold is:
\begin{enumerate}
\item geodesically convex, 
\item non-orientable,
\item non isochronous (time-sense not conserved) in the sense of Calabi--Markus (see \cite{CaMa}, for the exact terminology) 
\end{enumerate}
even dimensional Lorentzian manifolds.
\end{theorem}

\begin{proof}
By Proposition \ref{P:Lo} we know that $S$ is a Lorentzian manifold. 
The first (1) follows from the statement in Chensov \, \cite{Ch64}.  For (2), we can use the knowledge developed and acquired above. Indeed, since  $\R\cP^n$ is not orientable for $n$ even and applying the fact that a cartesian product of manifolds is orientable iff both manifolds are orientable it follows that since $S$ is of even dimension, $S$ is a non-orientable manifold. Now, applying the section 3 from \cite{CaMa} this implies that $S$ is geodesically convex and  time-like non orientable.
\end{proof}

\begin{corollary}
The fourth Frobenius manifold is uniquely determined by an orientable 2-fold covering. 
\end{corollary}
\begin{proof}
This is clear from elementary topology that we have $\R\cP^n\cong  S^n/\sim$ where $\sim$ is the antipodal map. So, we have the fiber bundle  
$S^n\to \R\cP^n$, with group $\Z_2$ of isometries $\pm I$. Another way of considering this to apply Section 3 of \cite{CaMa}  and in particular the ingredients constituting the proof of Theorem 3 in \cite{CaMa}.
\end{proof}

\begin{proposition}\label{P:Pierce}
The space of probability distributions $S$  is a non-orientable Lorentzian manifold, decomposed into pseudo-Riemannian submanifolds, being symmetric to each other with respect to the Pierce mirror. This Pierce mirror being an Hermitian hyperquadric is an ellipsoid and it isomorphic to $\mathbb{R}^n\times S^{n-1}$. This is the mirror symmetry of the fourth Frobenius manifold. 

\end{proposition}
\begin{proof}
The space of probability distributions $S$ is by Theorem \ref{T:nono} a non-orientable Lorentzian manifold. %By duality $\cH$ is also a non-orientable Lorentz manifold. 
Now, $S$ has a pair of pseudo-Riemannian submanifolds (by Collary \ref{C:sbmfd}).

Using Corollary \ref{C:proj} we have that $S$ is identified to a paracomplex projective space. Now, since the spin factor algebra has a pair of idempotents, this implies that  there exists a Pierce mirror, inducing symmetries of the space. 
We define the following morphism from the algebra of paracomplex numbers to $S^{2n}_1$. The Pierce mirror in the algebra corresponds to an involution in $S^{2n}_1$. The set of fixed points under this involution coincides with a hyperquadric. 

More precisely, applying the Theorem\, \ref{T:mirr} (paragraph 4.2), it turns out that this hyperquadric is an Hermitian ellipsoid hyperquadric, and that this is the set of fixed points under the symmetry. Therefore, it is a mirror. Consider the distance from hyperquadric to both domains (totally geodesic submanifolds of the fourth Frobenius manifold). From the properties of the module over a paracomplex algebra and its real interpretation (see section 1.3), we see that the distances from this hyperplane to both domains are geometrically the same (see Theorem \ref{T:mirr} and Theorem 4.3 in \cite{Roz97}). Now, since both domains are isomorphic to each other, the hyperquadric is a reflection mirror.  
This shows the statement.
\end{proof}

\section{Conclusion}
We have shown previously that the structure of the cone $\cC$ is a Lorenztian structure. Naturally, this leads to raising questions around {\it causality}, where the causality is interpreted here in the sense of S. W. Hawking  $\&$ J. F. R. Ellis  \cite{HaEl}. 

\smallskip 

Let us remark that  for a space constructed from modules over an algebra, the principle of causality is no longer a hypothesis. 
Such a space can be seen as a ``space-time'' {\it only if} we presuppose some {\it causality principle}. Indeed, in this way one can define a notion of time or {\it times}. 

\smallskip 

On the other side, we have a cone $\cC$ of measures of bounded variations. This is related to objects being central in machine learning and statistics. 
Machine learners and statisticians have translated the philosophical idea of {\it causality} into a viable inferential tool \cite{Sha96}. The main question researchers posed over the past thirty years was the extent to which a change in a causal variable might influence changes in a collection of effect variables, on the basis of observing an uncontrolled idle system. Traditionally, causal inference methods rely on a prespecified set of problem variables and use tools from counterfactual analysis, structural equation models and graphical models. 

\smallskip 

To conclude, this leads to a defining a bridge, between causality as defined by S. W. Hawking  $\&$ J. F. R. Ellis in \cite{HaEl} and causality defined for probability and statistics. Furthermore, this raises many questions and developments around these very active areas of research.

%%%%%%%%%%%%%%%%%%
%\appendix 

\end{document}